\documentclass[a4paper,11pt,reqno]{amsart}

\usepackage[margin=3cm]{geometry}

\usepackage{amssymb,amsfonts}
\usepackage{mathrsfs}

\usepackage{graphicx}
\usepackage{subcaption}
\usepackage{multirow}
\usepackage{booktabs}

\usepackage{xcolor}
\usepackage{textcomp}
\usepackage{adjustbox}  
\usepackage{placeins}   

\usepackage{natbib}

\newtheorem{theorem}{Theorem}
\newtheorem{proposition}[theorem]{Proposition}
\newtheorem{lemma}[theorem]{Lemma}
\newtheorem{corollary}[theorem]{Corollary}

\newtheorem{example}{Example}
\newtheorem{remark}{Remark}
\newtheorem{definition}{Definition}
\newtheorem{assumption}{Assumption}

\usepackage{enumitem}
\usepackage{multicol}

\DeclareMathOperator*{\minimize}{minimize}
\DeclareMathOperator*{\maximize}{maximize}
\DeclareMathOperator{\st}{subject\; to}
\DeclareMathOperator{\sta}{s.t.}
\def\ip#1#2{\left\langle #1, #2 \right\rangle}
\DeclareMathOperator*{\diag}{diag}
\DeclareMathOperator*{\supp}{supp}
\DeclareMathOperator{\CVaR}{CVaR}
\DeclareMathOperator{\EVaR}{EVaR}

\newcommand{\R}{\mathbb{R}}

\newcommand{\A}{\mathcal{A}}
\newcommand{\C}{\mathcal{C}}
\newcommand{\K}{\mathcal{K}}
\newcommand{\LL}{\mathcal{L}}
\newcommand{\PP}{\mathcal{P}}
\newcommand{\M}{\mathcal{M}}
\newcommand{\Ss}{\mathcal{S}}
\newcommand{\RR}{\mathcal{R}}

\let\origmaketitle\maketitle
\def\maketitle{
	\begingroup
	\def\uppercasenonmath##1{} 
	\let\MakeUppercase\relax 
	\origmaketitle
	\endgroup
}

\usepackage{hyperref}

\hypersetup{
  colorlinks = true,
  linkcolor  = blue,
  citecolor  = blue,
  urlcolor   = blue
}


\begin{document}

\title[Cooperation in Conic Programming]{\LARGE Cooperation in Conic Programming with Applications to Control, Production, \\ and Portfolio}

\author[M. Mart\'inez-Ant\'on \MakeLowercase{and} J. Puerto]{
{\large Miguel Mart\'inez-Ant\'on$^{\ddagger}$ and Justo Puerto$^{\ddagger}$}\bigskip\\
$^\ddagger$Institute of Mathematics (IMUS), Universidad de Sevilla\bigskip\\
\texttt{mmartinez31@us.es, puerto@us.es}
}

\maketitle

\begin{abstract}
    We introduce the class of cooperative conic games, a new family of transferable utility games whose characteristic function can be computed by solving a conic optimization problem. This framework unifies a broad range of optimization-based cooperative games within a common mathematical formulation, encompassing linear, second-order cone, semidefinite, and other convex nonlinear optimization problems admitting conic representations. Exploiting the structural properties of conic programs and conic duality theory, we derive general conditions ensuring the existence of core payoff allocations (or cost shares) and show how such allocations can be computed directly from optimal dual solutions. In particular, the proposed methodology replaces combinatorial verification procedures for core nonemptiness with tractable optimization-based criteria. The expressive power of the framework is illustrated through several representative applications. We show that it models cooperation in multi-agent optimal control, generalizes linear production games to nonlinear production settings such as Cobb--Douglas models, introduces cooperative formulations for portfolio selection under Markowitz, conditional value-at-risk, and entropic risk measures, and unifies a broad class of combinatorial optimization games through exact conic reformulations. These examples involve the most relevant cones in modern convex optimization, highlighting both the versatility of the proposed framework and its potential to strengthen the interaction between mathematical programming and cooperative game theory.
\end{abstract}

\section{Introduction}\label{sec:intro}

The growing interest in analyzing cooperation within different mathematical optimization models has led to an increasing number of studies addressing various aspects of this type of interaction. The interplay between mathematical programming and game theory has become a fruitful area of research. Classical models in mathematical programming deal with problems in which a decision maker seeks to design a program that optimizes the operation of a complex system. In contrast, game theory is concerned with decision problems involving several decision makers. Clearly, modeling and solving complex systems in which multiple decision makers interact require both the models and methodologies of mathematical programming and game theory. In this context, many classes of cooperative \emph{transferable utility} (TU) games arise, for instance, in linear production~\citep{Owen75, granot1986generalized, Timmer00, FFGP05}, network design~\citep{granot1981minimum, granot1984core}, facility location~\citep{GoSku04}, routing~\citep{Tamir88}, regression~\citep{Pinter2011}, queueing systems~\citep{Timmer13}, and more general combinatorial optimization problems~\citep{CaLe10, Deng99}, among many others.

The present analytical study of these situations is based on associating cooperative TU games with a general class of convex optimization problems. The objective is to derive payoff allocations or cost shares that are easily interpretable and, therefore, acceptable to the participating agents. There are several reasons for addressing the sharing problem, as it is a common concern in many fields, including public service, engineering, system analysis and design, economics, operations research, and management science.

Although, in general, testing the existence of a \emph{stable} (core) payoff allocation is NP-hard, finding good allocations is essential for enforcing this cooperation. Therefore, it is important to explore from different methodological standpoints the existence of stable allocations, in the sense that no coalition can raise objections regarding the share they have received.

This paper elaborates on this direction. Our goal is to introduce a novel class of cooperative games whose characteristic function is given by the optimal value of a conic program. Conic programming is a broad class of convex optimization problems that unifies linear, second-order cone, and semidefinite programming. It provides a unifying framework for representing nonlinear convex sets through structured constraints that define the solution set as the intersection of affine subspaces and cones. Many nonlinear optimization problems arising in applications can be expressed, or approximated arbitrarily well, using combinations of a small family of canonical cones, often referred to as the \emph{big five}, namely the nonnegative orthant, the second-order cone, the positive semidefinite cone, the power cone, and the exponential cone. This allows one to transform nonlinear convex optimization problems into structured conic programs solvable via interior-point methods implemented in modern conic optimization solvers. The interested reader is referred to \cite{BenTalNemirovski2001,BoydVandenberghe2004, anjos2011handbook, blekherman2012semidefinite, LuYe16}; and the references therein for further details.

There are two main reasons motivating this new class of games. On the one hand, it provides a general framework for previously studied optimization and operations research games, as in the aforementioned references. On the other hand, it opens the door to new cooperative settings arising from nonlinear optimization models that, to the best of our knowledge, have not been considered before, although they are of significant interest in areas such as optimal control, production models, and portfolio selection.

\subsection*{Contributions and Organization}

The main contribution of this paper is twofold: 
\begin{enumerate}[label=(\roman*)]
    \item the \textbf{conceptual} contribution is the introduction of \emph{cooperative conic games} as a new class of TU games based on conic programming; and

    \item the \textbf{methodological} contribution is the derivation of existence and computation results for core allocations by exploiting the structural properties and duality theory of conic programming.
\end{enumerate}
This thereby links, once more, another branch of mathematical programming with game theory.

Besides its main contribution, this paper presents connections and applications to real-world situations.

\begin{enumerate}
\item It shows the expressive power of the framework in different cooperative settings in \textbf{multi-agent optimal control}.

\item It generalizes linear production games and allows for more complex \textbf{nonlinear production} frameworks, such as Cobb--Douglas models.

\item It introduces cooperation in \textbf{portfolio selection} and applies it to linear (e.g., conditional value-at-risk), quadratic (e.g., Markowitz), and entropic risk measures.

\item It provides conditions for stability in very general \textbf{combinatorial optimization games}.
\end{enumerate}

We would like to emphasize that, throughout the paper, the reader will find instances involving the most relevant cones in the mathematical optimization literature, namely the nonnegative orthant, the second-order cone, the positive semidefinite cone, the power cone, and the exponential cone, as well as the sum of squares cone, the moment cone, the completely positive cone, and the copositive cone. This highlights the versatility of the proposed framework and is key to promoting its use among both theorists and practitioners.

The remainder of the paper is organized as follows. Section \ref{sec:prelims} introduces the general class of cooperative conic games and proves properties of the game by means of conditions on the conic formulation. Section \ref{sec:control} presents how this framework models cooperation in multi-agent optimal control. Section \ref{s:clpg} expands the notion of linear production games and proves stability in this general context. It then shows how it applies to nonlinear production settings, such as Cobb--Douglas models. Section \ref{s:psg} is devoted to analyzing the issue of cooperation in portfolio selection, studying three different cases for the risk measure: linear (conditional value-at-risk), quadratic (Markowitz), and entropic (entropic value-at-risk). These are three representative examples that can be easily extended to much of the existing literature. Section \ref{s:mbncq} generalizes combinatorial optimization games and proves stability under some conditions in this very general setting. The paper ends in Section \ref{s:cr} with some concluding remarks on the results of the paper and some ideas for further research on the topic.

\section{Cooperative Conic Games}\label{sec:prelims}

A \emph{transferable utility cooperative game} (TU game, for short) is a pair $(N,\nu)$ where $N$ is a finite set of \emph{players} or \emph{agents}, and $\nu:2^N \to \mathbb{R}$ is the \emph{characteristic function}, which assigns to each \emph{coalition} $S \subseteq N$ the profit (or cost) $\nu(S) \in \mathbb{R}$ generated by the cooperation of the players in $S$, without the participation of the remaining players in $N \setminus S$, and which can be arbitrarily allocated among its members; that is, utility is transferable. Players are allowed to cooperate, but their objective is to maximize their own individual benefit, and the set of all players $N$ is referred to as the \emph{grand coalition}. Notice that a similar definition can be adapted to cost TU games.

One of the main challenges in TU games consists of allocating the profit that the grand coalition can generate among the different players. Formally, an \emph{allocation} is a vector $z \in \mathbb{R}^N$, where $z_i$ denotes the share assigned to player $i$, for every $i \in N$. Two minimal requirements for an admissible solution are efficiency and individual rationality. An allocation $z \in \mathbb{R}^N$ is said to be \emph{efficient} if it properly distributes the profit of the grand coalition, i.e., $\sum_{i \in N} z_i = \nu(N)$. Moreover, an allocation is said to be \emph{individually rational} if each player receives at least what it can obtain on its own, that is, $z_i \geq \nu(\{i\})$ for all $i \in N$. Any allocation satisfying these two requirements is called an \emph{imputation}.

In this context, a \emph{solution concept} is a mapping that assigns to each TU game a set of imputations. Traditionally, two families of solutions can be distinguished.

\subsection*{Point-valued solutions} Firstly, we have the \emph{point-valued solution} concept, which arises when there is a single choice of an imputation within the solution set, so that each game is assigned a single imputation as its solution. The most widely used point-valued solution is the \emph{Shapley value}~\citep{shapley1953value}. Given a TU game $(N,\nu)$, its Shapley value is defined componentwise by
\begin{equation}\tag{Shapley}
   \frac{1}{|\mathscr S(N)|}\sum_{\sigma\in \mathscr S(N)} [\nu(P^\sigma_i\cup\{i\})-\nu(P^\sigma_i)], \quad \forall i\in N.
\end{equation}
Here, $\mathscr S(N)$ stands for the set of orders of $N$ and $P^\sigma_i:=\{j \in N: \sigma(j)< \sigma(i)\}$ denotes the subset of predecessors of player $i$ by the order $\sigma$. The number $\nu(P^\sigma_i\cup\{i\})-\nu(P^\sigma_i)$ accounts the marginal contribution of the player $i$ to the coalition of its predecessors by $\sigma$. Therefore, the Shapley value distributes to each player the average of its marginal contribution to the coalitions in which it does not participate.

\subsection*{Set-valued solutions and balancedness} These solution concepts make each TU game correspond to a solution set within all imputations. The most well-known representative of this type is the \emph{core}~\citep{gillies1953some}. Given a TU game $(N,\nu)$, its core is defined as
\begin{equation}\tag{Core}
    \left\{z\in \R^N : \quad \sum_{i\in S} z_i\geq \nu(S) \; \forall S\subset N, \quad \sum_{i\in N}z_i=\nu(N)\right\}.
\end{equation}
An allocation $z\in\R^N$ such that $\sum_{i\in S} z_i\geq \nu(S)$, for all coalition $S$, is said to be \emph{coalitionally rational}. Notice that this condition expands the individual rationality to all coalitions (not only singletons). Hence, the core of a TU game can be seen as the set of efficient and coalitionally rational allocations. Such core allocations are said to be \emph{stable} since no coalition can complain about the sharing of the grand coalition. By the Bondareva--Shapley theorem~\citep{bondareva1963some, shapley1967core}, a game is \emph{balanced} when its core is nonempty, and \emph{totally balanced} if the core of any of its subgames is nonempty.

\subsection*{Conic programming}  Let $V$ be a real vector space. Let $\mathcal{K} \subset V$ denote a pointed, closed, convex cone with nonempty interior; in this case, $\mathcal{K}$ is said to be a \emph{proper cone}. Consider the space of real linear mappings $V^* := \{y : V \to \mathbb{R} \mid \text{linear}\}$ as the dual vector space of $V$, together with the \emph{pairing} $\langle y, x \rangle_V := y(x)$. We define $\mathcal{K}^* := \{y \in V^* : \langle y, x \rangle_V \geq 0,\ \forall x \in \mathcal{K}\}$ as the dual cone, and denote by $\mathcal{K}^\circ$ the interior of $\mathcal{K}$. Recall that $x \in \mathcal{K}^\circ$ if and only if, for any $y \in \mathcal{K}^*$, the condition $\langle y, x \rangle = 0$ implies $y = 0$.

Now consider two real vector spaces $V$ and $U$, and a linear mapping $\mathcal{A} : V \to U$. Recall that the \emph{adjoint mapping} of $\mathcal{A}$ is defined as the unique linear mapping $\mathcal{A}^* : U^* \to V^*$ that satisfies

\begin{equation*}
    \ip{\A^*y}{x}_V=\ip{y}{\A x}_U, \; \forall x\in V, y\in U^*.
\end{equation*}

\begin{definition}[Conic programs]\label{def:cp}
    Let $V$ and $U$ be real vector spaces, $\A:V\to 
    U$ be a linear mapping, $\K\subset V$ be a proper cone, $b\in U$, and $c\in V^*$. The standard form \emph{conic programming {\rm(CP-P)} primal problem} and its \emph{dual problem {\rm (CP-D)}} are given by:
{    
\begin{multicols}{2} 
\noindent 
\begin{minipage}{0.48\textwidth}
\begin{align*}\label{eq:CP-P}\tag{CP-P}
    \minimize \quad  &\; \ip{c}{x}_V\\
    \st \quad &\; \A x= b,\\
                        &\; x\in \K;
\end{align*}
\end{minipage}

\columnbreak 

\noindent
\begin{minipage}{0.48\textwidth}
\begin{align*}\label{eq:CP-D}\tag{CP-D}
    \maximize \quad &\; \ip{y}{b}_U\\
    \st \quad &\; c-\A^* y\in \K^*.
\end{align*}
\end{minipage}
\end{multicols}
}
Here, $x$ is said to be the \emph{decision vector} and $y$ is said to be the \emph{dual price vector}.
\end{definition}

The starting point should be the geometric interpretation.  Note that the feasible sets are defined by the intersection of an affine subspace and a proper cone.
Classic examples of the above model are the standard linear programming where $\K=\mathbb{R}^n_+$ is the nonnegative orthant; second-order cone programming where $\K=\LL^n:=\{(x, t)\in \mathbb{R}^n\times \R: t\geq \|x\|_2 \}$ is the Lorentz cone; semidefinite programming where $\K=\Ss^n_+$ is the cone of $n\times n$ positive semidefinite matrices; or $\K=\K_1\times \K_2 \times \cdots \times \K_p$ being $\K_i$, $i=1,\ldots, p$  proper cones. Conic programming is a relatively novel area of mathematical programming that can be solved efficiently, and it is already implemented in modern solvers such as MOSEK.

Table~\ref{tab:cones} gathers the most relevant cones in convex optimization. The reader will find that all of them are used throughout the paper. Here, $\Ss^n$ stands for the space of $n\times n$ real symmetric matrices, and $\R_d[x]$ stands for the space of real multivariate polynomials of degree at most $d$.

\begin{table}
\footnotesize
\renewcommand{\arraystretch}{1.5}
\centering
\caption{Most relevant cones in convex optimization. All of them appear in this paper. \label{tab:cones}}
\begin{tabular}{p{3.8cm}p{1cm}p{1cm}p{6.4cm}}
\toprule
\textbf{Name} & $\K$ & \textbf{Space} & \textbf{Description}\\\midrule
Nonnegative orthant & $\R^n_+$ & $\R^n$ & $x\geq 0$ (componentwise order) \\
Second-order cone & $\LL^n$ & $\R^{n+1}$ & $\|x\|_2 \leq t$ (Euclidean norm epigraph)\\
Semidefinite cone & $\Ss^n_+$ & $\Ss^n$ & $x^\top Xx\geq 0, \forall x\in \R^n$ ($X\succeq 0$ L\"owner order) \\
Power cone & $\PP^n_\lambda$ & $\R^{n+1}$ & $|t|\leq \prod_{i=1}^n x_i^{\lambda_i}$ (geometric mean hypograph)\\
Exponential cone & $\K_{\exp}$ & $\R^3$ & cl $z\geq y\exp(x/y)$ (exp. perspective epigraph) \\
Copositive cone & $\C_n$ & $\Ss^n$ & $x^\top Xx\geq 0, \forall x\in \R^n_+$\\
Completely positive cone & $\C^*_n$ & $\Ss^n$ & $X=\sum_{k=1}^Kz_kz_k^\top$ for some $z_k\in\R^n_+$\\
Sum of squares cone & $\Sigma_{d}[x]$ & $\R_{2d}[x]$ & $f=\sum_{k=1}^K q_k^2$ for some $q_k\in\R_d[x]$\\
Moment cone & $\M_d^+$ & $\R_{2d}[x]^*$ & $M_d(y)\succeq 0$ where $M_{d}(y)_{\alpha,\beta}=y(x^{\alpha + \beta})$\\\bottomrule
\end{tabular}
\end{table}

From an algebraic viewpoint, the weak duality property holds for conic programs.\\

\subsection*{Weak duality} For any couple of feasible solutions $x,y$ of \eqref{eq:CP-P} and \eqref{eq:CP-D}, respectively, it always holds that
\begin{equation*}
    \begin{split}
    \ip{c}{x}_V-\ip{y}{b}_U&=\ip{c}{x}_V-\ip{y}{\A x}_U\\
    &=\ip{c}{x}_V-\ip{\A^*y}{x}_V\\
    &=\ip{c-\A^* y}{x}_V\geq 0.
\end{split}
\end{equation*}

Despite the formal similarities, there are several differences between linear and general conic programming. Beyond weak duality, there is no strong duality in the general case, i.e., there may be a nonzero duality gap between optimal values of the primal-dual pair. Moreover, optimal solutions may not be attained, even if there is zero duality gap.\\

\subsection*{Strong duality}
\begin{enumerate}
    \item Let either \eqref{eq:CP-P} or \eqref{eq:CP-D} be infeasible, and furthermore the other be feasible and has a nonempty interior (\emph{Slater condition}). Then the other is unbounded.

    \item Let \eqref{eq:CP-P} and \eqref{eq:CP-D} be both feasible, and furthermore one of them has a nonempty interior. Then there is no duality gap at optimality between them, i.e.,
    \begin{equation*}
        \inf\{\ip{c}{x}_V \mid \A x=b, x\in \K\}=\sup\{\ip{y}{b}_U \mid c-\A^* y\in \K^*\}.
    \end{equation*}

    \item If there is a pair of feasible solutions $x\in \K^\circ$ and $y$ (with $c-\A^*y\in (\K^*)^\circ$) of \eqref{eq:CP-P} and \eqref{eq:CP-D}, respectively. Then, they achieve the same optimal value (zero duality gap), and there exist optimal feasible solutions $x^\star$ and $y^\star$ such that $\ip{c}{x^\star}_V=\ip{y^\star}{b}_U$.
\end{enumerate}

\begin{definition}[Cooperative conic game]\label{def:ccg}
 A TU game $(N,\nu)$ is said to be a TU cooperative conic game (TU conic game, for short) if its characteristic function can be computed by solving a conic program. It means that for each coalition $S\subseteq N$ there exists parameters $(c_S, b_S, \A_S, \K_S)$ such that 
 \begin{align*}\label{eq:ccg}\tag{TU-CP}
    \nu(S) = \max\; (\min)  &\; \ip{c_S}{x}_V\\
    \st \; (\sta) &\; \A_S x= b_S,\\
                        &\; x\in \K_S.
\end{align*}
\end{definition}

Denoting the affine space $L(S):=\{x \in V: \A_S x = b_S\}$, the feasible set is then defined as the intersection $L(S)\cap \K_S$.

A TU game is said to be \emph{monotonic} if $S\subseteq T$ implies $\nu(S)\leq \nu(T)$.

\begin{lemma} \label{th:monotonicity}
Let a TU conic game be such that the vector $c_S$ satisfies $\supp(c_S)=B$ and $c_S|_B=c$ for all $S \subseteq N$. Moreover, for each $S \subseteq T$ and every feasible solution $x \in L(S) \cap \mathcal{K}_S$, there exists $w \in L(T) \cap \mathcal{K}_T$ such that $w|_B = x|_B$. Hence, the game is monotonic.
\end{lemma}

\begin{proof}
    Let $S\subseteq T$ and $x^\star$ be an optimal solution to the conic program \eqref{eq:ccg} for $S$. Since $x^\star$ is a feasible solution, by hypothesis exists $w$ feasible to the conic program for $T$. Hence, $$\nu(T)\geq \ip{c_T}{w}=\ip{c}{w|_B}=\ip{c}{x^\star|_B}=\ip{c_S}{x^\star}=\nu(S),$$
    where the first inequality comes from feasibility, the last equality comes from optimality, and the intermediate equalities come from the hypothesis on the objective function.
\end{proof}

A TU game  is said to be \emph{superadditive} if for all $S,T\subseteq N$ such that $S\cap T = \emptyset$ then $\nu(S\cup T)\geq \nu(S)+\nu(T)$. 

\begin{lemma}\label{th:subadditivity}
Let a TU conic game be such that the vector $c_S$ satisfies $\supp(c_S)=B$ and $c_S|_B=c$ for all $S \subseteq N$. Moreover, for each $S, T \subseteq N$ with $S \cap T = \emptyset$, and for feasible solutions $x_S \in L(S) \cap \mathcal{K}_S$ and $x_T \in L(T) \cap \mathcal{K}_T$, there exists $w \in L(S \cup T) \cap \mathcal{K}_{S \cup T}$ such that $w|_B = x_S|_B + x_T|_B$. Hence, the game is superadditive.
\end{lemma}

\begin{proof}
    Let $S\cap T = \emptyset$ and $x_S^\star, x^\star_T$ be optimal solutions to the conic program \eqref{eq:ccg} for $S$ and $T$, respectively. By hypothesis exists $w$ such that it is feasible to the conic program for $S\cup T$. Hence, 
    \begin{multline*}
        \nu(S\cup T)\geq \ip{c_{S\cup T}}{w}=\ip{c}{w|_B}=\ip{c}{x_S^\star|_B+x_T^\star|_B}\\=\ip{c}{x_S^\star|_B}+ \ip{c}{x_T^\star|_B}=\ip{c_S}{x_S^\star} + \ip{c_T}{x_T^\star}=\nu(S)+\nu(T),
    \end{multline*}
    where the first inequality comes from feasibility, the last equality comes from optimality, and the intermediate equalities come from the hypothesis on the objective function.
\end{proof}

These definitions (and results) naturally extend to cost TU conic games. We conclude this section with an illustrative example in polynomial games.

\begin{example}[Cooperative polynomial games]
    Despite (competitive) polynomial games having been studied in the literature~\citep[see, e.g.,][]{dresher2016polynomial, parrilo2006polynomial, laraki2012semidefinite}, we consider in this example cooperative games, where each agent $i \in N$ is represented by a variable $x_i$ and the goal of the grand coalition is to minimize the cost defined by a polynomial function $f \in \mathbb{R}_{2d}[x]$, where $x = (x_i)_{i \in N}$. Then, the cost generated by the cooperation of the players in $S$, without the participation of the remaining players in $N \setminus S$, is obtained by minimizing $f_S \in \mathbb{R}_{2d}[x]$, defined by the substitution $x_i = 1$ for all $i \in N \setminus S$. Then, the characteristic function $\nu(S)$ of the game can be represented and computed by the following primal-dual pair of conic programs~\cite[see][Section 5.2]{lasserre2009moments}:
{
    \begin{multicols}{2} 
\noindent 
\begin{minipage}{0.48\textwidth}
\begin{align*}
    \minimize \quad  &\; \ip{f_S}{y}_{\R_{2d}[x]^*}:=y(f_S)\\
    \st \quad &\; 1y:=y(1) = 1,\\
                        &\; y\in \M_d^+;
\end{align*}
\end{minipage}

\columnbreak 

\noindent
\begin{minipage}{0.48\textwidth}
\begin{align*}
    \maximize \quad &\; \ip{1}{\lambda}_\R:=\lambda\\
    \st \quad &\; f_S-\lambda \in \Sigma_d[x].
\end{align*}
\end{minipage}
\end{multicols}
}
    Then, for each coalition $S\subseteq N$ there exist parameters $c_S=f_S\in \R_{2d}[x]$, $b_S=1\in \R$, $\A_S=1:\R_{2d}[x]^*\to \R$ with $y\mapsto y(1)$, and $\K_S=\M_d^+$ whose dual cone is $(\M_d^+)^*=\Sigma_d[x]$, such that $\nu(S)$ is the optimal value of the corresponding conic program. Thus, this cooperative polynomial game can be seen as a monotonic cost TU conic game. Here, $y$ represents the decision vector and $\lambda$ represents the dual price vector.
\end{example}

The remainder of the paper is devoted to demonstrating how the proposed framework facilitates the analysis of cooperative behavior in several significant contexts arising in mathematical programming, engineering, economics, and operations research.

\section{Cooperation in Multi-Agent Optimal Control}\label{sec:control}

Consider the class of distributed linear time-invariant (LTI) differential systems composed of a finite set $N$ of interconnected subsystems, each governed by a local controller or agent. The dynamics of agent $i \in N$ are given by the following model:
\begin{gather*}
\frac{d}{dt}x_i(t) = A_{ii}x_i(t) + B_{ii}u_i(t) + w_i(t),\\
w_i(t) = \sum_{j \in N \setminus \{i\}} \left[A_{ij}x_j(t) + B_{ij}u_j(t)\right],
\end{gather*}
where $x_i : \mathbb{R}_+ \to \mathbb{R}^{q_i}$ is the state of agent $i$, $u_i : \mathbb{R}_+ \to \mathbb{R}^{r_i}$ is its corresponding control input, and $A_{ij} \in \mathbb{R}^{q_i \times q_j}$, $B_{ij} \in \mathbb{R}^{q_i \times r_j}$ for all pairs $i,j \in N$ denote the state-transition and input-to-state matrices, respectively. The term $w_i : \mathbb{R}_+ \to \mathbb{R}^{q_i}$ represents the influence of other agents on agent $i$.
Each subsystem is controlled by a different agent that has access only to its own state $x_i$ and chooses the value of its corresponding control input $u_i$.

Given a coalition of agents $S \subseteq N$, the agents in $S$ exchange all the information required to coordinate their actions. In other words, the coalition $S$ jointly optimizes the cost of its members by choosing their control inputs cooperatively, that is, $u_S(t) := (u_i(t))_{i \in S}$. In this case, the dynamics of the coalition are described by
\begin{gather*}
\frac{d}{dt}x_S(t) = A_{S}x_S(t) + B_{S}u_S(t) + w_S(t),\\
w_S(t) = \sum_{j\in N\setminus S} \left[A_{Sj}x_j(t) + B_{Sj}u_j(t)\right],
\end{gather*}
where $x_S(t) := (x_i(t))_{i \in S}$ is the aggregate state of the subsystems in $S$, $A_S := (A_{ij})_{i,j \in S}$, $B_S := (B_{ij})_{i,j \in S}$; $A_{Sj} := (A_{ij})_{i \in S}$ and $B_{Sj} := (B_{ij})_{i \in S}$ for all $j \in N \setminus S$.

The output of each agent is defined as follows:
\begin{equation*}
    z_i(t) = C_{i}x_i(t) + D_{i}u_i(t)
\end{equation*}
satisfying $D_i^\top D_i$ is regular and $D_i^\top C_i=0$. Analogously, the output of $S\subseteq N$ is given by
\begin{equation*}
    z_S(t) = C_{S}x_S(t) + D_{S}u_S(t),
\end{equation*}
where $C_S=\diag(C_i)_{i\in S}$, and $D_S=\diag(D_i)_{i\in S}$.

From a centralized viewpoint, given an initial state $x(0)$, the control problem can be posed as the following linear-quadratic (LQ) optimal control problem:
\begin{equation}\label{eq:lq}\tag{LQ}
    \minimize_{u_N(t)\in \R^R} \quad J(u_N,x(0)) = \int_0^\infty z_N(t)^\top z_N(t) dt,
\end{equation}
where $R:=\sum_{i\in N}r_i$.

The cooperative games arising in optimal distributed control follow the same general rationale. The objective of every coalition is to minimize the global output energy of the centralized system. Accordingly, the output energy achieved by a coalition of players, without the participation of the remaining ones, is given by the best performance attainable under its restricted control strategy. In this context, two families of optimal control games arise: (i) games in which the players are the agents and the coalition controller is restricted to act on a linear subspace; and (ii) games in which the action of the controllers is determined by the links of a communication network. In the latter case, the set of edges $C$ of the complete graph $(N,C)$ constitutes the set of players.

The remainder of this section is due to prove that both classes of games are cost TU conic games.

\subsection{Acting under Coalitional Restriction}

We consider the first family of games. In general, the objective of a coalition $S$ is to minimize the overall output energy subject to a coalitionally restricted control action. Formally, for each $S \subseteq N$, the global controller $u_N$ is constrained to belong to a linear subspace $U_S \subseteq \mathbb{R}^{R}$, that is, $u_N : \mathbb{R}_+ \to U_S$, with the property that $S \subseteq T$ implies $U_S \subseteq U_T$ and $U_N = \mathbb{R}^{R}$. Suppose that $u_N(t) = Kx_N(t)$. Since the control input is a linear function of the state, this is referred to as a \emph{state-feedback} controller, and the matrix $K$ is called the \emph{state-feedback gain}. This yields a closed-loop LTI system. To ensure that the restricted control action is feasible, it is necessary to assume the existence of: (i) a state-feedback gain $K$ such that ${\rm Im}(K) \subseteq U_S$ and that stabilizes the overall system; and (ii) a quadratic Lyapunov function $f(\xi)=\xi^\top P_S\xi$ for the closed-loop system that provides an upper bound on the output energy. This situation induces a cooperative game $(N,\nu)$, where 
\begin{equation*}
    \nu(S) \geq \min_{u_N} \left\{ J(u_N,x(0)) = \int_0^\infty z_N(t)^\top z_N(t) dt \; : \; u_N(t)\in U_S\right\},
\end{equation*}
is the best upper bound on the output energy that can be certified by quadratic Lyapunov functions under the restricted controller $u_N : \mathbb{R}_+ \to U_S$.

\begin{theorem}\label{th:acting}
     The game $(N,\nu)$ is a monotonic TU conic game. Furthermore, the characteristic function can be represented by the following semidefinite program {\rm (SDP):}
     \begin{subequations}\label{model1}
     \begin{align}
    \nu(S) = \min_{Q\succ 0, Y} \quad & \gamma \\
    \sta \quad & Y=\Pi_SY, \label{ctr:proj} \\
    & \begin{bmatrix}
   \gamma & x(0)^\top\\
   x(0) & Q
\end{bmatrix} \succeq 0, \label{ctr:inverse}\\
    & \begin{bmatrix}
   A_{N}Q + QA_{N}^\top + B_{N}Y + Y^\top B_{N}^\top & (C_NQ + D_NY)^\top\\
   C_NQ + D_NY & -I
\end{bmatrix} \preceq 0\label{ctr:stab},
\end{align}
\end{subequations}
where $\Pi_S$ stands for the orthogonal projection matrix onto the space $U_S$.
\end{theorem}

\begin{proof}
    Let $(Q^\star, Y^\star, \gamma^\star)$ be an optimal solution of problem \eqref{model1}. By the Schur complement we have that the linear matrix inequality (LMI) \eqref{ctr:stab} implies
\begin{equation*}
        A_{N}Q^\star + Q^\star A_{N}^\top + B_{N}Y^\star + (Y^{\star})^\top B_{N}^\top + (C_NQ^\star + D_NY^\star)^\top(C_NQ^\star + D_NY^\star) \preceq 0.
    \end{equation*}

    With $P^\star=(Q^{\star})^{-1} \succ 0$ and left- and right-multiplying with $P^\star$. Then we have
 \begin{equation*}
        (A_N + B_NY^\star P^\star)^\top P^\star + P^\star (A_N+B_NY^\star P^\star) + (C_N+D_NY^\star P^\star)^\top(C_N+D_NY^\star P^\star) \preceq 0.
    \end{equation*}

    Now, taking the state-feedback gain $K^\star=Y^\star (Q^{\star})^{-1}=Y^\star P^\star$,
 \begin{equation*}
        (A_N + B_NK^\star)^\top P^\star + P^\star (A_N+B_NK^\star) + (C_N+D_NK^\star)^\top(C_N+D_NK^\star) \preceq 0.
\end{equation*}

    Then, it follows that $x(0)^\top P^\star x(0)$ is an upper bound on the output energy $J(u_N,x(0))$ for $u_N(t)=K^\star x(t)$~\citep[see][]{boyd1994linear}. 

Second, by Schur complement and optimality, we have that the LMI \eqref{ctr:inverse} implies $\gamma^\star = x(0)^\top (Q^{\star})^{-1}x(0)=x(0)^\top P^\star x(0)=\nu(S)$ is the best upper bound via quadratic functions. 

Finally, by \eqref{ctr:proj}$, K^\star = Y^\star (Q^{\star})^{-1} = \Pi_S Y^\star (Q^{\star})^{-1}=\Pi_S K^\star$, what implies $u_N(t)\in U_S$, what concludes the proof.
    
\end{proof}

\begin{theorem}
    Given $S$, let $\Pi_S$ be the orthogonal projection matrix onto the space $U_S$. Then $(N,\nu)$ is equal to the \eqref{eq:lq} optimal control game $(N,\nu_{\rm LQ})$ where
$$
\nu_{\rm LQ} (S) := \min_{u_N(t)\in \R^R} J(u_N,x(0))
$$
with the input data $(A_N, B_S:=B_N\Pi_S, C_N, D_S:=D_N\Pi_S)$. Moreover, $\nu(S)$ has the following simpler conic representation in primal {\rm SDP} standard form:
\begin{align*}
    \nu(S) = \min_{Z\in \Ss^n} \quad & \ip{C^\top C}{Z_n} + \ip{D_S^\top D_S}{Z_b}\\
    \sta  \quad & Z_u^\top B_S^\top + B_SZ_u + Z_nA^\top + AZ_n = -x_0x_0^\top\\
    & Z := \begin{bmatrix}
        Z_b & Z_u \\
        Z_u^\top & Z_n
    \end{bmatrix} \succeq 0.
\end{align*}
\end{theorem}

\begin{proof}
    We prove that problem \eqref{model1} and the problem
\begin{align}
   \nu_{\rm LQ}(S) = \min_{Q\succ 0, Y} \quad & \gamma\label{model2}\\
    \sta \quad & \begin{bmatrix}
   \gamma & x(0)^\top\\
   x(0) & Q
\end{bmatrix} \succeq 0, \nonumber\\
    & \begin{bmatrix}
   A_{N}Q + QA_{N}^\top + B_{S}Y + Y^\top B_{S}^\top & (C_NQ + D_SY)^\top\\
   C_NQ + D_SY & -I
\end{bmatrix} \preceq 0,\nonumber
\end{align}
    are equivalent. First, assume a feasible solution $(Q,Y,\gamma)$ for \eqref{model1}, then 
\begin{equation*}
    \begin{bmatrix}
   A_{N}Q + QA_{N}^\top + B_{N}\Pi_SY + Y^\top\Pi_S B_{N}^\top & (C_NQ + D_N\Pi_S Y)^\top\\
   C_NQ + D_N\Pi_S Y & -I
\end{bmatrix} \preceq 0
\end{equation*}
so $(Q,Y,\gamma)$ is feasible for \eqref{model2}. Proceeding backwards, assume $(Q, Y, \gamma)$ is feasible for \eqref{model2}. Then, we see $(Q, \Pi_S Y,\gamma)$ is feasible for \eqref{model1}. The two LMIs \eqref{ctr:inverse}-\eqref{ctr:stab} are verified, and $\Pi_S^2Y = \Pi_S Y$ holds due to the property of the orthogonal projection. Since a feasible solution of one problem induces a feasible solution of the other, and vice versa, with the same objective value, the two problems are equivalent.

Problem \eqref{model2} solves the corresponding \eqref{eq:lq} optimal control problem because it is unconstrained on the control input, so the already announced equality with $\nu_{\rm LQ}$ holds.

Finally, we can use the primal-dual pair of semidefinite programs associated with the algebraic Riccati equation (that provides the solution to \eqref{eq:lq}, see \cite{willems1971least}):
{
\setlength{\columnsep}{1.3cm}
\begin{multicols}{2} 
\noindent 
\begin{minipage}{0.4\textwidth}
\begin{align*}
     \min_{Z\in\Ss^n} & \ip{C^\top C}{Z_n} + \ip{D_S^\top D_S}{Z_b}\\
    \sta \; & Z_u^\top B_S^\top + B_SZ_u + Z_nA^\top + AZ_n = -x_0x_0^\top,\\
    & Z:=\begin{bmatrix}
        Z_b & Z_u \\
        Z_u^\top & Z_n
    \end{bmatrix} \succeq 0;
\end{align*}
\end{minipage}

\columnbreak

\noindent
\begin{minipage}{0.4\textwidth}
\begin{align*}
    \max_{P\in\Ss^n} & x(0)^\top Px(0)\\
    \sta  & \begin{bmatrix}
   D_S^\top D_S & B_S^\top P\\
   PB_S & A^\top P + PA + C^\top C
\end{bmatrix}  \succeq 0;
\end{align*}
\end{minipage}
\end{multicols}
}
to compute the characteristic function $\nu_{\rm LQ}(S)=\nu(S)$.
 \end{proof}

This result shows that it is equivalent to restrict the action and to operate freely under a constrained control system, leading to simpler and diverse formulations for the same TU conic game.

\begin{example}[Activation/deactivation of controllers]
    A special case of the above game arises when, for each coalition $S$, only the controllers in $S$ are active, while the controllers outside $S$ are switched off. In this case, the linear subspace $U_S$ corresponds to the projection onto the variables $r_i$, $i \in S$, and the associated orthogonal projection matrix $\Pi_S$ is a diagonal matrix with ones in the rows indexed by $r_i$, $i \in S$, and zeros elsewhere.
\end{example}

\subsection{Acting under Network-Driven Communication}
Consider the complete graph $(N,C)$, where the set of players is identified with the set of edges $C$. In this framework, given a coalition $\Lambda \subseteq C$, we assume that the agents can only coordinate through a network whose physical topology is described by the graph $(N,\Lambda)$. A necessary and sufficient condition for any two agents to coordinate their control is that they belong to the same connected component of $(N,\Lambda)$; that is, there exists a path of links connecting them. The notion of connectedness induces a partition of the set $N$ into disjoint connected components. This partition is denoted by $N/\Lambda \subset 2^N$, where $S \cap T = \emptyset$ for $S,T \in N/\Lambda$ and $\bigcup_{S \in N/\Lambda} S = N$.

A feasible multi-agent control scheme induced by $(N,\Lambda)$ must satisfy that there exists a state-feedback gain $K_\Lambda$ adapted to its communication constraints that stabilizes the overall system, together with a quadratic Lyapunov function $f(\xi) = \xi^\top P_\Lambda \xi$ for the closed-loop system that provides an upper bound on the output energy and is also consistent with the communication constraints.

Under these conditions, one can define a cooperative game $(C,\nu)$ where the value $\nu(\Lambda)$ of each coalition $\Lambda$ is given by the best upper bound on the output energy that can be certified via quadratic Lyapunov functions of the multi-agent control scheme induced by $(N,\Lambda)$.

This multi-agent control scheme and its connection with cooperative game theory were introduced in \cite{maestre2014coalitional} for distributed LTI difference systems, which constitute the discrete-time analogue of our differential system.

    \begin{theorem}
        The game $(C,\nu)$ is a monotonic TU conic game. Furthermore, the characteristic function can be represented by the following {\rm SDP:}
     \begin{align*}
    \nu(\Lambda) = \min_{Q\succ 0, Y} \quad & \gamma\\
    \sta \quad & Y=\sum_{S\in N/\Lambda} \Pi'_S Y\Pi_S,\\
    & Q=\sum_{S\in N/\Lambda} \Pi_S Q\Pi_S,\\
    & \begin{bmatrix}
   \gamma & x(0)^\top\\
   x(0) & Q
\end{bmatrix} \succeq 0, \\
    & \begin{bmatrix}
   A_{N}Q + QA_{N}^\top + B_{N}Y + Y^\top B_{N}^\top & (C_NQ + D_NY)^\top\\
   C_NQ + D_NY & -I
\end{bmatrix} \preceq 0,
\end{align*}
where $\Pi_S$ (respectively, $\Pi'_S$) denotes the block-diagonal projection matrix with identity blocks corresponding to $q_i$ (respectively, $r_i$), $i\in S$, and zeros elsewhere.
    \end{theorem}

    \begin{proof}
Let $(Q^\star, Y^\star, \gamma^\star)$ be an optimal solution of the above program.  By the properties of $\Pi_S$ we have that $$(Q^{\star})^{-1}=\left(\sum_{S\in N/\Lambda}\Pi_SQ^\star\Pi_S\right)^{-1}=\sum_{S\in N/\Lambda}\Pi_S(Q^{\star})^{-1}\Pi_S$$ and therefore 
\begin{multline*}
    K^\star = Y^\star (Q^{\star})^{-1} = \left(\sum_{S\in N/\Lambda} \Pi'_S Y\Pi_S\right) \left( \sum_{S\in N/\Lambda}\Pi_S(Q^{\star})^{-1}\Pi_S\right)\\=\sum_{S\in N/\Lambda} \Pi'_S Y^\star (Q^{\star})^{-1}\Pi_S=\sum_{S\in N/\Lambda} \Pi'_S K^\star \Pi_S,
\end{multline*}
what implies $u_N(t)=K^\star x_N(t)$ is coherent with the communication multi-agent control scheme. 

The reader can easily complete the proof by verbatim following the proof of Theorem~\ref{th:acting}.
    \end{proof}

\begin{remark}
    These games are generally not subadditive; however, in the context of multi-agent optimal control, this is not a significant concern, as the objective is not necessarily to form the grand coalition, but rather to assess the value associated with different groups of agents. Nevertheless, it is possible to impose additional conditions on the Lyapunov matrices to guarantee subadditivity of the game. More specifically, for disjoint coalitions $S \cap T = \emptyset$, one may require $P_S + P_T \succeq P_{S \cup T}$. Furthermore, concavity can be imposed via $P_S + P_T \succeq P_{S \cup T} + P_{S \cap T}$ for all coalitions $S,T$. Recall that concavity in these games implies core nonemptiness and therefore balancedness, as well as stability of the Shapley value. These assumptions ensure that each coalition can bound the overall output energy in such a way that the resulting cooperative game is subadditive or concave, while remaining consistent with the system dynamics.

\end{remark}

\section{Cooperation in Production \label{s:clpg}}

This section is devoted to a broad class of special TU conic games that extend and generalize the celebrated linear production games \citep{Owen75}. Let $N$ be a set of agents, each endowed with a bundle of $m$ resources, represented by the mapping $b : N \to \mathbb{R}^m$. The resources themselves have no intrinsic value, that is, there is no primary demand for them. There is, however, a derived demand for these resources, as they can be used to produce goods that can be sold at given market prices. Let $x$ denote the decision vector of production levels, belonging to a proper cone $\mathcal{K}$; let $p$ denote the vector of market prices; and let $\mathcal{A}$ be the $m$-dimensional real linear operator that maps each production plan to its resource requirements.

Each agent $i \in N$ is interested in maximizing the utility of its resource endowment by solving the following pair of primal-dual conic programs:
{
\begin{multicols}{2} 
\noindent 
\begin{minipage}{0.48\textwidth}
\begin{align*}
    \maximize \quad  &\; \ip{p}{x}\\
    \st \quad &\; \A x= b(i),\\
                        &\; x\in \K;
\end{align*}
\end{minipage}

\columnbreak 

\noindent
\begin{minipage}{0.48\textwidth}
\begin{align*}
    \minimize \quad  &\; \ip{y}{b(i)}\\
    \st \quad &\; \A^* y-p\in \K^*.
\end{align*}
\end{minipage}

\end{multicols}
}

If the agents of a coalition $S\subseteq N$ cooperate, then they can pool their resources, namely $b(S):=\sum_{i\in S} b(i)$, and the problem to be considered by the coalition is given by:
{
\begin{multicols}{2}
\noindent 
\begin{minipage}{0.48\textwidth}
\begin{align*}\label{eq:conicpg-primal}\tag{\rm CPG-P}
    \maximize \quad  &\; \ip{p}{x}\\
    \st \quad &\; \A x= b(S),\\
                        &\; x\in \K;
\end{align*}
\end{minipage}

\columnbreak 

\noindent
\begin{minipage}{0.48\textwidth}
\begin{align*}\label{eq:conicpg-dual}\tag{\rm CPG-D}
    \minimize \quad &\; \ip{y}{b(S)}\\
    \st \quad &\; \A^* y - p \in \K^*.
\end{align*}
\end{minipage}

\end{multicols}
}

The natural questions that arise are the following: under which conditions are the agents in $N$ willing to cooperate, and how should the surplus generated by their cooperation be allocated among them? These questions can be addressed by introducing a cooperative game in which the worth of a coalition of agents $S$ is given by the value $\nu(S)$, representing the maximum reward that the members of $S$ can obtain by acting jointly.

We define the general \emph{conic production game} $(N,\nu)$ as the TU conic game whose characteristic function $\nu(S)$ is given by the optimal value of problem \eqref{eq:conicpg-primal}. Accordingly, we introduce the class of conic production games with $N$ players as the set of 5-tuples $(N,p,\mathcal{A},b,\mathcal{K})$. Each element of this class defines a cooperative game $(N,\nu)$ whose characteristic function is given, for every $S \subseteq N$, by

\begin{align*}\label{eq:conicpg}\tag{\rm CPG}
   \nu(S) :=  \max \quad  &\; \ip{p}{x}\\
    \sta \quad &\; \A x= b(S),\\
                        &\; x\in \K.
\end{align*}

Notice that the class of linear production games is a subclass of conic production games given by the $5$-tuple $(N, p, \A, b,\R^n_+)$.
The class of conic production games with $N$ players on the same cone is a cone itself. Indeed, the game with characteristic function $\lambda \nu$ corresponds to the 5-tuple $(N, \lambda p, \A, b,\K)$ which clearly is a conic production game. However,  adding two games $(N, p^1, \A^1, b^1,\K)$ and $(N, P^2, \mathcal{A}^2, b^2,\K)$ may simply turn out to be an infeasible problem for the corresponding characteristic function:
\begin{align*}
 (\nu^1+\nu^2)(S)=\max \quad & \ip{p^1+p^2}{x}\\
    \sta \quad &\; \A^1 x= b^1(S),\\
    & \; \A^2 x= b^2(S),\\
                        &\; x\in \K.
\end{align*}
Since the constraints of both conic problems could be incompatible. This implies that the class of conic production games is a cone but not necessarily convex. Nevertheless, if the constraints defining each one of the original games are compatible, then $(N, p^1+p^2, \mathcal{A}^{1}\times \mathcal{A}^2, b^1 \times b^2, \K)$ is another conic production game. 

Some additional structural properties of this class of games include monotonicity, superadditivity, and, most importantly, total balancedness. We will show that, under mild assumptions, every conic production game $(N,\nu)$ is totally balanced; that is, the core of the game and the cores of all its subgames are nonempty.

\begin{corollary}
The conic production game $(N,\nu)$ is monotonic and superadditive.
\end{corollary}
\begin{proof}
The result follows from applying Lemmas~\ref{th:monotonicity} and \ref{th:subadditivity}.
\end{proof}

\begin{theorem}\label{tg:production}
Let $(N,\nu)$ be the conic production game defined by $(N,p,\A,b,\K)$. If the conic program~\eqref{eq:conicpg} for the grand coalition $N$ verifies the Slater condition, then the core of $(N,\nu)$ is nonempty.
\end{theorem}
\begin{proof}
Under the hypothesis of the theorem, by strong duality, there exists an optimal dual price vector $y_N^\star$ such that $\nu(N)=\ip{y_N^\star}{b(N)}$. Let us define the allocation $z_i=\ip{y_N^\star}{b(i)}$ for all $i\in N$. The efficiency is already provided by strong duality. Besides, for any subset $S\subset N$ the vector $y_N^\star$ is feasible for its dual program. Hence, $$\sum_{i\in S} z_i = \sum_{i\in S} \ip{y_N^\star}{b(i)} = \ip{y_N^\star}{b(S)} \geq \nu(S)$$ where the last inequality follows from weak duality. Thus, $z=(z_i)_{i\in N}$ belongs to the core of $(N,\nu)$ proving the claim.
\end{proof}

The reader can easily check that if the conic program \eqref{eq:conicpg} for a coalition $S$ verifies the Slater condition, the same proof also applies to any subgame induced by a subset of players $S$, which proves that any conic production game $(N,\nu)$ under these mild conditions is totally balanced.

\subsection{Cobb--Douglas Production Game}

The Cobb--Douglas production model, originally introduced by \cite{cobb1928theory}, is one of the most widely used models in economics to describe the relationship between production ($P$), labor ($L$), and capital ($K$). In its classical form, production is modeled as $P=\kappa L^\lambda K^{1-\lambda},$ where $\kappa>0$ is a productivity parameter, and $\lambda \in [0,1]$ is the \emph{elasticity} of labor. The elasticity parameter quantifies the responsiveness of production to changes in the input levels: a $1\%$ increase in labor results in a $\lambda\%$ increase in production, while a $1\%$ increase in capital yields a $(1-\lambda)\%$ increase in production. More generally, given $d$ goods with levels $x_j \geq 0$, $j=1,\ldots,d$, the Cobb--Douglas production function is defined by $f(x)=\prod_{j=1}^d x_j^{\lambda_j}$, where $\sum_{j=1}^d\lambda_j=1$ and $\lambda_j\geq 0$ represents the elasticity of good $j$.

In the context of cooperation, let $N$ be a set of agents, each endowed with a bundle of $m$ resources $b : N \to \mathbb{R}^m$. Let $\mathcal{A}$ be the $m$-dimensional real linear operator that maps production levels to resource consumption.

If the agents of a coalition $S\subseteq N$ cooperate, then they can pool their resources, namely $b(S):=\sum_{i\in S} b(i)$, and seek to solve the following utility-maximization problem:

\begin{equation}\tag{Cobb--Douglas} \label{pro:cobb-douglas} 
\nu(S): = \max_{x\geq 0} \left\{ \prod_{j=1}^d x_j^{\lambda_j} \; : \; \A x = b(S)\right\}.
\end{equation}

We can define the \emph{Cobb--Douglas production game} $(N,\nu)$ as the TU game whose characteristic function is given by the optimal value of \eqref{pro:cobb-douglas}.

\begin{proposition}\label{prop:cobb-douglas}
    The Cobb-Douglas production game $(N,\nu)$ is a conic production game.
\end{proposition}

\begin{proof}
    The proof follows from the fact that the hypograph of the Cobb--Douglas function is exactly the power cone $\PP^d_\lambda$. Therefore, the Cobb--Douglas production game can be represented by the following power cone program {\rm (PCP):}
 \begin{align}\label{eq:cd}\tag{CD-PCP}
    \nu(S) = \max \quad & t \nonumber \\
    \sta \quad & \A x = b(S),\nonumber\\
    & (x,t)\in \PP^d_\lambda,\nonumber
    \end{align}
    what proves the claim.
\end{proof}

The class of Cobb--Douglas production games is a subclass of conic production games given by the $5$-tupla $(N, (0,1), (\A, 0), b, \PP^d_\lambda)$. The dual program of \eqref{eq:cd} is given by:
\begin{equation}
   \minimize \quad  \ip{y}{b(S)} \quad  \st \quad  (\A^\top y,1) \in (\PP^d_\lambda)^*,\nonumber
    \end{equation}
where the expression for the dual cone of the power cone  $(\PP^d_\lambda)^*$ is given by $$\left\{(x,t)\in \R_+^d\times\R : |t|\leq \prod_{j=1}^d \left(\frac{x_j}{\lambda_j}\right)^{\lambda_j}\right\}.$$
\begin{corollary}
    Let $(N,\nu)$ be a Cobb--Douglas production game. If the conic program~\eqref{eq:cd} for the grand coalition $N$ verifies the Slater condition, then the core of $(N,\nu)$ is nonempty.
\end{corollary}

\begin{proof}
    By means of Proposition~\ref{prop:cobb-douglas}, the claim follows from Theorem \ref{tg:production} applied to the Cobb--Douglas production game. Given an optimal solution $y^\star$ of the dual of \eqref{eq:cd} for the grand coalition $N$. The allocation defined by $z_i=\ip{y^\star}{b(i)}$ for $i\in N$ belongs to the core.
\end{proof}

\section{Cooperation in Portfolio Selection\label{s:psg}}

This section is devoted to examining cooperative behavior within the classical problem of portfolio optimization. We demonstrate that the resulting cooperative framework constitutes yet another subclass of cooperative conic games. Let $N$ be a set of agents. Each agent $i \in N$ owns a capital $c_i$ that they are willing to invest in a collection of assets $j=1,\ldots,m$, where asset $j$ has expected return $\mu_j$. Let $x \in \mathbb{R}_+^m$ denote the vector of portfolio allocations. Each agent seeks to maximize the expected return on a portfolio selected from the $m$ available assets. Furthermore, we assume that all agents agree on a common risk measure, represented by the function $\mathcal{R}(x)$, which depends on the portfolio selection.

Under this framework, a coalition $S \subseteq N$, with total capital $c(S)=\sum_{i \in S} c_i$, seeks to solve the following return-maximization problem subject to a risk constraint:
\begin{align}
\maximize \quad & \ip{\mu}{x}  \tag{Portfolio} \nonumber \\
\st \quad & \sum_{j=1}^m x_j =c(S), \nonumber\\
& \RR(x)\le R(S), \nonumber \\
& x\ge 0. \nonumber
\end{align}
Here, $R(S)\geq 0$ denotes a common admissible risk threshold for the coalition $S$.

As we shall see, this general framework gives rise to different portfolio optimization models depending on the choice of the risk measure. In the following, we present three representative cases, namely the Markowitz portfolio model, the conditional value-at-risk model, and the entropic value-at-risk model.

\subsection{Markowitz Portfolio Game}

One of the most influential models in finance and modern operations research is the classical Markowitz portfolio model \citep{Markowitz52} and its many variants.

In this setting, we consider the general framework described above with the quadratic risk measure $\mathcal{R}^2(x)=x^\top Qx$, where $Q \in \mathbb{R}^{m \times m}$ is a positive definite covariance matrix. Under cooperation, the agents in a coalition $S$ pool their capital, that is, $c(S)=\sum_{i \in S} c_i$, and agree to assume a common level of risk bounded by the threshold $R^2(S)$. Under these assumptions, the decision problem faced by coalition $S$ gives rise to the following characteristic function:
\begin{align}\label{eq:markowitz}\tag{Markowitz}
\nu(S) := \max \quad & \ip{\mu}{x} \\
\sta \quad & \sum_{j=1}^m x_j =c(S), \nonumber\\
& x^\top  Q x \le R^2(S), \nonumber \\
& x\ge 0. \nonumber
\end{align}

Although this model has been extensively studied, our focus is on the cooperation among agents within this framework. We show that, after a suitable reformulation, the resulting \emph{Markowitz portfolio game} $(N,\nu)$ can be cast as a cooperative conic game.

\begin{lemma}
    The Markowitz portfolio game $(N,\nu)$ is a TU conic game. Furthermore, the characteristic function can be represented by the following second-order cone program {\rm (SOCP):}
 \begin{align}\label{eq:socp}\tag{M-SOCP}
     \nu(S) =  \max \quad & \ip{\mu}{x}\nonumber \\
    \sta \quad & \begin{bmatrix}
        0 & 0 & e^\top  \\
        -I & 0 & L  \\
        0 & 1 & 0 
    \end{bmatrix}\begin{bmatrix}
        y \\
        s \\
        x \\
    \end{bmatrix} = \begin{bmatrix}
        c(S)\\
        0 \\
        R(S)
    \end{bmatrix},\nonumber\\
    &  (y,s; x) \in \LL^m \times \R^{m}_+,\nonumber
\end{align}
where $Q=L^\top L$ is the Cholesky decomposition of $Q$ and $e$ is a suitable vector of ones.
\end{lemma}

\begin{proof}
    The covariance matrix $Q$ admits a Cholesky factorization $Q=L^\top L$ where $L$ is a triangular matrix. Thus, $x^\top Qx= y^\top  y$ with $y=Lx$ and problem \eqref{eq:markowitz} becomes
\begin{align*}
\maximize \quad & \ip{\mu}{x}   \\
\st \quad & \sum_{j=1}^m x_j =c(S), \\
& Lx-y=0, \\
& s = R(S), \\
& \sum_{j=1}^m y_j^2 \le s^2, \\
& s,x\ge 0, y \in \mathbb{R}^m.\nonumber
\end{align*}

The claim follows then from the observation that $\sum_{j=1}^m y_j^2\leq s^2$ together with $s\geq 0$ is equivalent to $(y,s)\in \LL^m$, and rewriting the program in matrix form.
\end{proof}

The dual program of \eqref{eq:socp} for coalition $S$ is:
\begin{align*}
    \minimize \quad  &\; c(S)u + R(S)w\\
    \st \quad & L^\top v + ue \geq \mu, \\
    &\; (v,w) \in \LL_m.
\end{align*}

In general, the game $(N,\nu)$ is not balanced or superadditive unless some conditions are required on the risk threshold mapping $R(\cdot)$. In the following, we will assume a sufficient condition on $R$ that ensures both superadditivity and balancedness.

\begin{assumption}\label{ass}
It is natural to assume that the larger the coalition, the greater the average risk that the agents are willing to accept. Formally, if $S,T\subseteq N$ coalitions such that $|S|\leq |T|$ then
\begin{equation*}
\frac{R(T)}{|T|} \ge \frac{R(S)}{|S|}.
\end{equation*}
\end{assumption}

\begin{proposition}
Under Assumption \ref{ass}, the Markowitz portfolio game $(N,\nu)$ is superadditive.
\end{proposition}
\begin{proof}
We only have to prove that $R(T)+R(S)\le R(S\cup T)$. Without loss of generality, assume that $|S|\le |T|$. Recall that for any nonnegative real numbers $a,b,c,d$, then $a/b\ge c/d$ if and only if $a/b\ge (a+c)/(b+d)\ge c/d$. Then applying this observation and the assumption, we have 
$$\frac{R(S\cup T)}{|S|+|T|} \ge \frac{R(T)}{|T|}\ge \frac{R(S)+R(T)}{|S|+|T|} \ge \frac{R(S)}{|S|}.$$
Therefore, $R(S\cup T) \ge R(S)+R(T)$. Now one has to apply Lemma \ref{th:subadditivity} to complete the proof.
\end{proof}

\begin{theorem}\label{th:markowitz}
Let $(N,\nu)$ be a Markowitz portfolio game under Assumption \ref{ass}. If the conic program \eqref{eq:socp} for the grand coalition $N$ verifies the Slater condition, then the core of $(N,\nu)$ is nonempty.
\end{theorem}
\begin{proof}
Under the hypothesis of the theorem, by strong duality, there exists an optimal dual price vector $(u_N^\star, v^\star_N, w^\star_N)$ such that $\nu(N)=\sum_{i\in N}c_iu_N^\star + R(N)w^\star_N$. Let us define the allocation $z_i=c_iu_N^\star + \frac{R(N)}{|N|}w^\star_N$ for all $i\in N$. The efficiency is already provided by strong duality. Besides, for any subset $S\subset N$ the vector $(u_N^\star, v^\star_N, w^\star_N)$ is feasible for its dual program. Hence, $$\sum_{i\in S} z_i = \sum_{i\in S} \left(c_iu_N^\star + \frac{R(N)}{|N|}w^\star_N\right) = c(S)u_N^\star + 
|S|\frac{R(N)}{|N|}w^\star_N \overset{(\small A \ref{ass})}{\geq} c(S)u_N^\star + R(S)w^\star_N \geq \nu(S)$$ where the last inequality follows from weak duality. Thus, $z=(z_i)_{i\in N}$ belongs to the core of $(N,\nu)$ proving the claim.
\end{proof}

\subsection{Conditional Value-at-Risk Portfolio Game}

Another classical portfolio optimization model considers the conditional value-at-risk (CVaR) of the portfolio return as the risk measure. To this end, suppose that $T$ different scenarios are considered. The probability of occurrence of scenario $t=1,\ldots,T$ is $p_t \geq 0$, with $\sum_{t=1}^{T} p_t = 1$. We assume that, for each asset $j=1,\ldots,m$, the realization of its random return $\rho_j$ under scenario $t$ is known and denoted by $r_{jt}$. The expected return of asset $j$ is therefore given by $\mu_j=\sum_{t=1}^T p_tr_{jt}$. The scenario-based formulation captures the dependence among the returns of the different assets. The return of a portfolio $x$ under scenario $t$ is $y_t=\sum_{j=1}^m r_{jt}x_j$. Given a confidence level $0 \leq \beta \leq 1$, the conditional value-at-risk of the portfolio return $y=(y_t)_{t=1}^{T}$ is defined as

$$
\CVaR_\beta(y) : = \frac{1}{1-\beta} \sum_{t=1}^{\lfloor (1-\beta)T\rfloor} p_{(t)}y_{(t)}
$$
where $y_{(1)}\geq y_{(2)}\geq \cdots \geq y_{(T)}$. Under cooperation, the agents in a coalition $S$ pool their capital, that is, $c(S)=\sum_{i\in S} c_i$, in order to maximize their expected return while constraining the risk measure $\mathcal{R}(y)=\CVaR_\beta(y)$ to be at most the coalitional risk threshold $R(S)$. Under these assumptions, the decision problem faced by coalition $S$ is given by the following optimization problem:

\begin{align}\label{eq:cvar}\tag{CVaR-P} 
\nu(S) := \max \quad & \ip{\mu}{x}\nonumber \\
\sta \quad & \sum_{j=1}^m x_j =c(S), \nonumber\\
& y_t=\sum_{j=1}^m r_{jt}x_j, \quad \forall t=1,\ldots, T;\nonumber\\
& \CVaR_\beta(y)\le R(S), \nonumber \\
& x\ge 0. \nonumber
\end{align}

The resulting CVaR \emph{portfolio game} $(N,\nu)$ can be cast as a cooperative conic game due to the following well-known  linear programming (LP) reformulation of \eqref{eq:cvar} \cite[see, e.g.,][]{rockafellar2000optimization, mansini2015linear}:
 \begin{align}\label{eq:cvar-lp}\tag{CVaR-LP}
\nu(S) =\max \quad & \ip{\mu}{x}  \nonumber \\
\sta \quad & \sum_{j=1}^m x_j =c(S), \nonumber\\
& y_t=\sum_{j=1}^m r_{jt}x_j, \quad \forall t=1,\ldots, T;\nonumber\\
& d^+_t \ge y_t-\eta, \nonumber \\
& \eta + \frac{1}{1-\beta}\sum_{t=1}^T p_td^+_t \le R(S), \nonumber \\
& d^+_t, x\ge 0. \nonumber
\end{align}

We say that the aggregated risk that the agents are willing to accept is \emph{additive} if $R(S)=\sum_{i\in S} R(i)$ where $R(i):=R(\{i\})$. If this equality holds for every coalition $S\subseteq N$, then we say that the risk threshold mapping $R$ is additive.

\begin{theorem}\label{th:cvar}
Let $(N,\nu)$ be a $\CVaR$ portfolio game. If the risk threshold mapping $R$ is additive or satisfies Assumption \ref{ass}, then $(N,\nu)$ is superadditive and its core is nonempty.
\end{theorem}

\begin{proof}
    Notice that in the additive case, a stable allocation is given by $z_i=c_iu^\star+R(i)w^\star$ for all $i\in N$, where $u^\star$ and $w^\star$ are optimal dual prices of \eqref{eq:cvar-lp} for the grand coalition $N$. The proof follows from the fact that $(N,\nu)$ is a conic (linear, indeed) production game and applying Theorem \ref{tg:production}. 
    
    Otherwise, if Assumption~\ref{ass} holds, then a stable allocation is given by $z_i=c_iu^\star+\frac{R(N)}{|N|}w^\star$. The proof follows the same arguments as that of Theorem~\ref{th:markowitz}.

    For the superadditivity, given $S,T\subseteq N$ such that $S\cap T = \emptyset$, observe that under the two hypotheses $R(S\cup T)\geq R(S) + R(T)$ holds. It remains only to apply Lemma \ref{th:subadditivity}.
\end{proof}

\subsection{Entropic Value-at-Risk Portfolio Game}

In the third portfolio optimization model under consideration, the risk measure is given by the entropic value-at-risk~\citep{ahmadi2019portfolio}. The reader is referred to the previous section for the details of the portfolio setting considered here. Given a confidence level $0 \leq \beta \leq 1$, the entropic value-at-risk (EVaR) of the portfolio return $y=(y_t)_{t=1}^{T}$ is defined as

\begin{equation*}
    \EVaR_\beta(y):=\min_{z\geq 0} z\ln \left(\frac{1}{1-\beta}\sum_{t=1}^T p_t\exp\left(\frac{y_t}{z}\right)\right).
\end{equation*}

The following result proves the conic representability of the entropic value-at-risk via an exponential conic lifting.

\begin{lemma}\label{lema:lifting}
    The following feasible set
    \begin{align}
        \{(y,\tau,z,v,w)\in \R^T\times \R_+\times \R_+\times \R_+^T\times \R_+ &: z\geq \sum_{t=1}^T p_tv_t \label{ctr:1}\\
        & w+z\ln\left(\frac{1}{1-\beta}\right) \leq \tau \label{ctr:2}\\
        & (y_t - w, z, v_t)\in \K_{\exp}, \quad \forall t=1,\ldots, T\}    \label{ctr:3}
        \end{align}
        is a conic lifting of $\EVaR_\beta$.
\end{lemma}

\begin{proof}
Let $(y,\tau)\in \R^T\times \R_+$ be in the epigraph of $\EVaR_\beta$. It means
    \begin{equation*}
        \EVaR_\beta(y)= \min_{z\geq 0} z\ln \left(\frac{1}{1-\beta}\sum_{t=1}^T p_t\exp\left(\frac{y_t}{z}\right)\right)\leq \tau.
    \end{equation*}
    Thus, there exists $z(y,\tau)$ such that $z(y,\tau)\ln \left(\frac{1}{1-\beta}\sum_{t=1}^T p_t\exp\left(\frac{y_t}{z(y,\tau)}\right)\right)\leq \tau$. Then, taking $w(y,\tau)=z(y,\tau)\ln\left(\sum_{t=1}^T p_t\exp\left(\frac{y_t}{z(y,\tau)}\right)\right)$, it is easy to check that the vector
    $$
    \left(y,\tau,z(y,\tau),z(y,\tau)\exp\left(\frac{y_t-w(y,\tau)}{z(y,\tau)}\right)_{t=1}^T,w(y,\tau)\right)\in \R^T\times \R_+\times \R_+\times \R_+^T \times \R_+
    $$
    belongs to the feasibility set \eqref{ctr:1}-\eqref{ctr:3}.

    Conversely, let $(y,\tau,z,v,w)\in \R^T\times \R_+\times \R_+\times \R_+^T \times \R_+$ be in the feasibility set; then we have the following chain of implications:

    \begin{align}
        z\overset{\eqref{ctr:1}}{\geq} \sum_{t=1}^T p_tv_t \overset{\eqref{ctr:3}}{\geq} & \sum_{t=1}^Tp_tz\exp\left(\frac{y_j-w}{z}\right)\nonumber\\
        1 \geq & \sum_{t=1}^Tp_t\exp\left(\frac{y_j-w}{z}\right)\nonumber\\
        0\geq & \ln\left(\sum_{t=1}^T p_t\frac{\exp\left(\frac{y_t}{z}\right)}{\exp\left(\frac{w}{z}\right)}\right)\nonumber\\
        0\geq & \ln\left(\sum_{t=1}^T p_t\exp\left(\frac{y_t}{z}\right)\right) - \ln\left(\exp\left(\frac{w}{z}\right)\right)\nonumber\\
        \frac{w}{z} \geq & \ln\left(\sum_{t=1}^T p_t\exp\left(\frac{y_t}{z}\right)\right)\nonumber\\
        w \geq& z\ln\left(\sum_{t=1}^T p_t\exp\left(\frac{y_t}{z}\right)\right).\label{ctr:4}
        \end{align}
Thus, it implies
        \begin{align*}
            z\ln \left(\frac{1}{1-\beta}\sum_{t=1}^Tp_t\exp\left(\frac{y_t}{z}\right)\right)&=z\ln \left(\sum_{t=1}^Tp_t\exp\left(\frac{y_t}{z}\right)\right) + z\ln\left(\frac{1}{1-\beta}\right)\\
            &\overset{\eqref{ctr:4}}{\leq} w + z\ln\left(\frac{1}{1-\beta}\right) \overset{\eqref{ctr:2}}{\leq} \tau.
            \end{align*}
Hence, 
            $$
            \EVaR_\beta(y)= \min_{z\geq 0} z\ln \left(\frac{1}{1-\beta}\sum_{t=1}^Tp_t\exp\left(\frac{y_t}{z}\right)\right)\leq \tau,
            $$
            and then $(y,\tau)$ belongs to the epigraph of $\EVaR_\beta$ as it was claimed.
\end{proof}

As in the CVaR portfolio game, under cooperation, the agents in a coalition $S$ pool their capital, that is, $c(S)=\sum_{i\in S} c_i$, in order to maximize their expected return while constraining the risk measure $\mathcal{R}(y)=\EVaR_\beta(y)$ to be at most the coalitional risk threshold $R(S)$. Under these assumptions, the decision problem faced by coalition $S$ is given by the following optimization problem:
\begin{align}
\nu(S) := \max \quad & \ip{\mu}{x}  \tag{EVaR-P} \nonumber \\
\sta \quad & \sum_{j=1}^m x_j =c(S), \nonumber\\
& y_t=\sum_{j=1}^m r_{jt}x_j, \quad \forall t=1,\ldots, T;\nonumber\\
& \EVaR_\beta(y)\le R(S), \nonumber \\
& x\ge 0. \nonumber
\end{align}

By means of the conic lifting, the resulting EVaR \emph{portfolio game} $(N,\nu)$ can be cast as a cooperative conic game.

\begin{corollary}
    The $\EVaR$ portfolio game $(N,\nu)$ is a TU conic game. Furthermore, the characteristic function can be represented by the following exponential cone program {\rm (ECP):}

\begin{align}\label{eq:evar-ecp}\tag{EVaR-ECP}
\nu(S) = \max \quad & \ip{\mu}{x}  \nonumber \\
\sta \quad & \sum_{j=1}^m x_j =c(S), \nonumber\\
& y_t=\sum_{j=1}^m r_{jt}x_j, \quad \forall t=1,\ldots, T;\nonumber\\
& w + z\ln\left(\frac{1}{1-\beta}\right)\leq R(S), \nonumber \\
& z\geq \sum_{t=1}^Tp_tu_t, \nonumber\\
& x\ge 0, \nonumber\\
& (y_t-w, z, u_t)\in \K_{\exp}, \quad \forall t=1,\ldots, T.\nonumber
\end{align}
\end{corollary}

\begin{proof}
    The result follows directly from Lemma~\ref{lema:lifting}.
\end{proof}

\begin{theorem}
Let $(N,\nu)$ be an $\EVaR$ portfolio game. If the risk threshold mapping $R$ is additive or satisfies Assumption \ref{ass}, then $(N,\nu)$ is superadditive. Moreover, if the conic program \eqref{eq:evar-ecp} for the grand coalition $N$ satisfies the Slater condition, then the core of $(N,\nu)$ is nonempty.
\end{theorem}

\begin{proof}
    The proof follows the same arguments as that of Theorem~\ref{th:cvar}, but using \eqref{eq:evar-ecp} and its dual prices.
\end{proof}

\section{Mixed-Binary Quadratic Optimization Games\label{s:mbncq}}

In this section, we analyze a broad class of cooperative games that encompasses a wide range of combinatorial optimization games.

Let us assume that there exists a finite set $N$ of players and that we are given, for each coalition of players $S\subseteq N$, parameters $(Q_S, c_S, b_S, A_S, B_S)$ such that the characteristic function of a game $\nu(S)$ is given by the optimal value of the following mixed-binary quadratic program:
\begin{align*}\tag{MBQOG}
\nu(S) := \min \quad & x^\top  Q_S x +2 c_S^\top  x \\
\sta \quad & A_S x  \leq b_S, \\
& x\ge 0, \; x_j\in \{0,1\}, \; \forall j\in B_S. 
\end{align*}
In such a way, we define the TU game $(N,\nu)$ as a \emph{mixed-binary quadratic optimization game.}

\begin{lemma}\label{th:burer}
    Assume that the linear slice $L(S):=\{x\in \R^{n_S}_+ : A_S x = b_S\}$ satisfies $L(S)\subset \{x\in \R^{n_S}: 0\leq x_j\leq 1, \forall j\in B_S\}$ for all $S\subseteq N$. Then, the game $(N,\nu)$ is a TU conic game. Furthermore, its characteristic function can be represented by the following completely positive program {\rm (CPP):}
    \begin{align}\label{eq:cpp}\tag{MBQOG-CPP}
    \nu(S) = \min \quad & \ip{Q_S}{X} + 2 c_S^\top  x \nonumber\\
    \sta \quad & A_S x  \leq b_S, \nonumber\\
    & \A_S X \leq b_S \circ b_S, \nonumber\\
    & X_{jj} = x_j, \; \forall j\in B_S,\nonumber\\
    & \begin{bmatrix}
        1 & x^\top \\
        x & X
    \end{bmatrix}\in \C^*_{n_S+1},\nonumber
    \end{align}
where $\A_S:=\left((a_S)_i(a_S)_i^\top\right)_{i\in m_S}$ and $\circ$ stands for the Hadamard product.
\end{lemma}

\begin{proof}
    The result is a direct consequence of \cite[Theorem 2.6]{Burer09}.
\end{proof}

The dual program of \eqref{eq:cpp} is the following copositive program (COP):
\begin{align}\label{eq:cop}\tag{MBQOG-COP}
\maximize \quad
&
-\alpha
-\ip{u}{b_S}
-\ip{v}{b_S\circ b_S}
\\[1ex]
\st\quad
& 2y = c_S+ A_S^\top u - \displaystyle\sum_{j\in B_S} w_j e_j, \nonumber\\
& Y = Q_S+\mathcal A_S^*(v)
+\displaystyle\sum_{j\in B_S} w_j E_{jj}, \nonumber\\
& \begin{bmatrix}
\alpha &
y^\top
\\[1ex]
y
&
Y
\end{bmatrix}
\in \mathcal C_{n_S+1}, \nonumber
\\[1ex]
&
u,v \in \R^{m_S}_+,
\quad
w\in\mathbb R^{|B_S|},\quad \alpha\in \R,\nonumber
\end{align}
where $e_j$ stands for the $j$th vector of the canonical basis, and $E_{jj}=\diag(e_j)$ for all $j\in B_S$.

The reader may observe that, as pointed out in \cite{Burer09}, the assumption in Theorem~\ref{th:burer} is not restrictive. Indeed, if $B_S \neq \emptyset$ and the condition is not already satisfied, it can be enforced without loss of generality by introducing, for each $j \in B_S$, the additional constraint $x_j+s_j=1$, where $s_j \geq 0$ is a slack variable.

Motivated by their practical relevance in combinatorial optimization games, we focus on the particularly important class of mixed-binary quadratic optimization games defined by $(Q,c,b(S),A,B)$ for each coalition $S$, where $b(S):=\sum_{i\in S} b(i)$ is obtained by pooling the right-hand-sides (resources) $b(i)$ contributed by the players in the coalition. 
Despite the generality of this class of games, we subsequently establish sufficient conditions guaranteeing balancedness.

\begin{theorem}
Let $(N,\nu)$ be the mixed-binary quadratic optimization game defined by $(Q,c,b(S),A,B)$ for each coalition $S$. If the conic program \eqref{eq:cpp} for the grand coalition $N$ verifies the Slater condition and has an optimal dual price vector \linebreak $(\alpha^\star, u^\star, v^\star, w^\star, y^\star, Y^\star)$. Then, the following claims hold.
\begin{enumerate}[label=\roman*)]
    \item If $y^\star\in \R_+^n$, then the core of $(N,\nu)$ is nonempty.

    \item If $\frac{1}{|N|}\geq \displaystyle\sup_{\substack{\eta\geq 0 \\ \eta^\top y^\star<0}}\frac{(\eta^\top y^\star)^2}{\alpha^\star \eta^\top Y^\star \eta}$, then the core of $(N,\nu)$ is nonempty.
\end{enumerate}

\end{theorem}
\begin{proof}
Let us define the allocation $z_i=-\frac{\alpha^\star_N}{|N|}-\ip{u_N^\star}{b(i)} - \sum_{j\in N}\ip{v^\star_N}{b_i\circ b_j}$ for all $i\in N$. Note that $$\sum_{i\in N} z_i = \sum_{i\in N} \left[-\frac{\alpha^\star_N}{|N|}-\ip{u_N^\star}{b(i)} - \sum_{j\in N}\ip{v^\star_N}{b_i\circ b_j}\right] = \nu(N)$$ so the efficiency is already provided.

To prove coalitional rationality in case i), first notice that under hypothesis, we have that the polynomial function 
\begin{equation}\label{eq:nonnegativity}
    f(\alpha, \xi,\eta)=[\xi\; \eta]\begin{bmatrix}
    \alpha & y^{\star \top}\\
    y^\star & Y^\star
\end{bmatrix}\begin{bmatrix}
    \xi \\ \eta
\end{bmatrix} = \alpha\xi^2 + 2\xi\eta^\top y^\star + \eta^\top Y^\star\eta \geq 0
\end{equation} for $\alpha,\xi,\eta\geq 0$ due to $Y^\star\in \C_n$ and $y^\star\in \R^n_+$.  Besides, for any subset $S\subset N$ the solutions $u_N^\star, v^\star_N$, and $\frac{|S|}{|N|}\alpha^\star_N$ are feasible for its dual program \eqref{eq:cop} by \eqref{eq:nonnegativity}. Hence, \begin{multline*}
    \sum_{i\in S} z_i = \sum_{i\in S} \left[-\frac{\alpha^\star_N}{|N|}-\ip{u_N^\star}{b(i)} - \sum_{j\in N}\ip{v^\star_N}{b_i\circ b_j}\right] \\= -|S|\frac{\alpha^\star_N}{|N|}-\ip{u_N^\star}{b(S)} - \sum_{i\in S}\sum_{j\in N}\ip{v^\star_N}{b_i\circ b_j} \\ \leq -|S|\frac{\alpha^\star_N}{|N|} -\ip{u_N^\star}{b_S} - \sum_{i\in S}\sum_{j\in S}\ip{v^\star_N}{b_i\circ b_j} \\= -|S|\frac{\alpha^\star_N}{|N|} - \ip{u^\star_N}{b_S} - \ip{v^\star_N}{b_S\circ b_S} \leq  \nu(S)
\end{multline*} 
where the last inequality follows from weak duality. Thus, $z=(z_i)_{i\in N}$ belongs to the core of $(N,\nu)$ proving the claim.

To prove ii), consider the change of variable $\alpha\mapsto \lambda\alpha^\star$ with $\lambda\geq 0$ in $f$. Calculating the partial derivative on $\xi$ an equalizing to $0$,
\begin{equation*}
    \frac{\partial f}{\partial \xi} = 2\lambda\alpha^\star\xi + 2\eta^\top y^\star = 0
\end{equation*}
provides a minimum (by convexity) in $\xi^*=-\frac{\eta^\top y^\star}{\lambda\alpha^\star}$. Then we have that $$f(\lambda\alpha^\star, \xi, \eta)\geq f(\lambda\alpha^\star, \xi^*, \eta)=\eta^\top Y^\star\eta - \frac{(\eta^\top y^\star)^2}{\lambda \alpha^\star}.$$

By hypothesis, $\frac{|S|}{|N|}\geq \frac{1}{|N|}\geq \frac{(\eta^\top y^\star)^2}{\alpha^\star \eta^\top Y^\star \eta}$ what implies $\eta^\top Y^\star\eta - \frac{(\eta^\top y^\star)^2}{\frac{|S|}{|N|} \alpha^\star}\geq 0$ for all $\eta\geq 0$ such that $\eta^\top y^\star< 0$ and each coalition $S$. Thus, using also the case i), $f\left(\frac{|S|}{|N|}\alpha^\star, \xi, \eta\right)\geq 0$ for all $\xi,\eta\geq 0$. The remainder of the proof follows equal to the case i).
\end{proof}

The significance of this theorem lies in replacing a combinatorial verification of balancedness with a dual-based analytical criterion. Rather than establishing core nonemptiness by evaluating an exponential family of coalition constraints, we derive a sufficient condition that reduces the task to verifying explicit requirements on the optimal dual prices associated with the conic reformulation.

We conclude this section by noting that the combinatorial optimization games most commonly examined in the literature \citep{Deng99,granot1981minimum,Tamir88,GoSku04} constitute special cases of this broader class, obtained by simply omitting the quadratic term from the objective function of the underlying characteristic function optimization problem.

\section{Conclusion\label{s:cr}}

This paper introduces the concept of cooperative conic games, providing a unified framework that connects cooperative game theory and conic programming. The proposed perspective encompasses a broad family of optimization-based cooperative games, including linear, second-order cone, semidefinite, and other convex nonlinear optimization problems admitting conic representations. This level of generality makes it possible to capture, within a single mathematical framework, a variety of cooperative settings that have either been studied independently or have not previously been considered from a cooperative game-theoretic approach.

The main methodological contribution of the paper is the exploitation of the structural properties of conic programs and conic duality to derive balancedness results and payoff allocation (or cost sharing) rules. In particular, the proposed framework provides sufficient conditions for core nonemptiness that can be verified through the optimal dual solution of the underlying conic program, replacing combinatorial verification procedures with tractable optimization-based criteria. Moreover, the dual variables naturally induce sharing rules that are both economically interpretable and computationally tractable.

To illustrate the versatility of the proposed framework, we have studied several representative classes of cooperative optimization games. We have shown how cooperative conic games model cooperation in multi-agent optimal control, encompassing both controller activation and communication network schemes. We have also generalized linear production games to nonlinear production settings, including Cobb--Douglas models. We explored cooperative behavior in portfolio selection under the Markowitz, conditional value-at-risk, and entropic risk measures. Finally, we unified a broad class of combinatorial optimization games through exact conic reformulations for binary quadratic optimization. Taken together, these examples demonstrate that the proposed framework accommodates the most relevant cones in modern convex optimization, highlighting both its modeling versatility and its applicability across a wide range of real-world problems arising in engineering, economics, management science, and operations research.

\subsection*{Future research} Several future research avenues naturally arise from this work. On the theoretical side, it would be interesting to investigate additional structural conditions ensuring stronger cooperative properties, such as convexity, and other solution concepts for broader classes of conic games. More generally, we believe that the proposed framework opens the way to the study of cooperation in many nonlinear optimization models that have not yet been analyzed from a cooperative game-theoretic perspective. For instance, our introduction of cooperative polynomial optimization games illustrates how the framework can be extended beyond the existing literature and opens up a promising research direction that deserves further investigation. 

The authors of this paper hope that these developments will continue to strengthen the interaction between mathematical programming and game theory.

\subsection*{Acknowledgements}

The authors would like to express their sincere gratitude to the Agencia Estatal de Investigaci\'on of the Ministerio de Ciencia, Innovaci\'on y Universidades for supporting this research through the grants PID2020-114594GB-C21, PID2024-156594NB-C21, Plan Estatal 2024-2027 - Retos - Programaci\'on Conjunta Internacional PCI2024-155024-2, Programa de Movilidad PRX24/00256, and IMUS--Mar\'ia de Maeztu CEX2024-001517-M. They also gratefully acknowledge the support of the Programa de Andaluc\'ia FEDER through project C‐EXP‐139‐UGR23.

\section*{Declarations}

\subsection*{Funding} This research was supported by grants PID2020-114594GB-C21, PID2024-156594NB-C21, PCI2024-155024-2, PRX24/00256, and IMUS--Mar\'ia de Maeztu CEX2024-001517-M funded by MCIN/AEI/10.13039/501100011033; and the project C‐EXP‐139‐UGR23 funded by Programa de Andaluc\'ia FEDER.

\subsection*{Conflict of interest/Competing interests} The authors declare that they have no known competing financial interests or personal relationships that could have appeared to influence the work reported in this paper.

\bibliographystyle{abbrvnat}
\bibliography{references}

\end{document}